\documentclass[11pt]{amsart}
\usepackage{amscd,amssymb}
\usepackage[plainpages,backref,urlcolor=blue]{hyperref}

\newtheorem{theorem}{Theorem}[section]
\newtheorem{corollary}[theorem]{Corollary}
\newtheorem{lemma}[theorem]{Lemma}
\newtheorem{prop}[theorem]{Proposition}

\theoremstyle{definition}

\newtheorem{example}[theorem]{Example}
\newtheorem{remark}[theorem]{Remark}
\newtheorem*{ack}{Acknowledgments}

\newcommand{\Z}{\mathbb{Z}}
\newcommand{\Q}{\mathbb{Q}}

\newcommand{\C}{\mathbb{C}}

\newcommand{\T}{\mathbb{T}}
\newcommand{\CP}{\mathbb{CP}}
\newcommand{\RP}{\mathbb{RP}}
\newcommand{\LL}{\mathbb{L}}

\newcommand{\A}{{\mathcal A}}

\DeclareMathOperator{\corank}{corank}

\DeclareMathOperator{\im}{im}
\DeclareMathOperator{\Hom}{Hom}

\DeclareMathOperator{\SO}{SO}

\DeclareMathOperator{\abf}{{abf}}

\DeclareMathOperator{\ord}{{ord}}
\DeclareMathOperator{\lcm}{{lcm}}
\DeclareMathOperator{\gr}{gr}
\DeclareMathOperator{\Lie}{Lie}

\newcommand{\Tors}{{\mathcal{T}}}

\newcommand{\h}{\mathfrak{h}}
\newcommand{\m}{\mathfrak{m}}
\newcommand{\wh}{\widehat{\h}}

\newcommand{\surj}{\twoheadrightarrow}

\newcommand{\isom}{\xrightarrow{\,\simeq\,}}

\newcommand{\abs}[1]{\left| #1 \right|}

\begin{document}
%\date{October 10, 2008}
%\date{December 8, 2009}

\title[Quasi-K\"{a}hler $3$-manifold groups]%
{Quasi-K\"{a}hler groups, $3$-manifold groups, and formality}

\author[A.~Dimca]{Alexandru Dimca$^{1,2}$}
\address{Laboratoire J.A.~Dieudonn\'{e}, UMR du CNRS 6621, 
Universit\'{e} de Nice--Sophia Antipolis, Parc Valrose,
06108 Nice Cedex 02, France}
\email{dimca@math.unice.fr}

\author[S.~Papadima]{Stefan Papadima$^{1,3}$}
\address{Institute of Mathematics Simion Stoilow, 
P.O. Box 1-764,
RO-014700 Bucharest, Romania}
\email{Stefan.Papadima@imar.ro}

\author[A.~Suciu]{Alexander~I.~Suciu$^{4}$}
\address{Department of Mathematics,
Northeastern University,
Boston, MA 02115, USA}
\email{a.suciu@neu.edu}

\thanks{$^{1}$Partially supported by the French-Romanian 
Programme LEA Math-Mode}
\thanks{$^{2}$Partially supported by ANR-08-BLAN-0317-02 (S\'{E}DIGA)}
\thanks{$^{3}$Partially supported by CNCSIS grant ID-1189/2009-2011}
\thanks{$^{4}$Partially supported by NSA grant 
H98230-09-1-0012 and an ENHANCE grant from 
Northeastern University}



\begin{abstract}
In this note, we address the following question:  Which 
$1$-formal groups occur as fundamental groups of 
both quasi-K\"{a}hler manifolds and closed, connected, 
orientable $3$-manifolds. We classify all such groups,
at the level of Malcev completions, and compute their coranks. 
Dropping the assumption on realizability by $3$-manifolds, we 
show that the corank equals the isotropy index of the 
cup-product map in degree one. Finally, we examine the formality 
properties of smooth affine surfaces and quasi-homogeneous 
isolated surface singularities. In the latter case, we describe 
explicitly the positive-dimensional components of the first 
characteristic variety for the associated singularity link.
\end{abstract}
\maketitle

\section{Introduction}
\label{sect:intro}

\subsection{K\"{a}hler groups and $3$-manifold groups}
\label{subsec:intro1}

The following question was raised by W.~Goldman and 
S.~Donaldson in 1989, and independently by A.~Reznikov 
in 1993:  Which $3$-manifold groups are K\"{a}hler groups? 
Both kinds of groups have a geometric origin: the first kind 
arise as fundamental groups of compact, connected, orientable 
$3$-dimensional manifolds (without boundary), while the 
second kind arise as fundamental groups of compact, 
connected, K\"{a}hler manifolds.  

In \cite{Re}, Reznikov obtained a deep restriction on certain groups 
lying at the intersection of these two naturally defined classes. 
In \cite{DS07}, two of us answered the question 
in full generality: those groups which occur as both 
$3$-manifold groups and K\"{a}hler groups
are precisely the finite subgroups of $\SO(4)$. 

In this note, we pursue the analogous problem, for the 
more general class of quasi-K\"{a}hler groups.  Under 
a $1$-formality assumption (satisfied by all K\"{a}hler groups), 
we classify the Malcev completions of $3$-manifold groups 
that fall into this class.  

\subsection{Quasi-K\"{a}hler groups}
\label{subsec:intro2}

A manifold $X$ is called {\em quasi-K\"{a}hler}\/ if 
$X=\overline{X} \setminus D$, where $\overline{X}$ is 
a compact, connected K\"{a}hler manifold and $D$ is a 
divisor with normal crossings.  If a group $G$ can be realized 
as the fundamental group of such a manifold $X$, we say $G$ 
is a quasi-K\"{a}hler group---or, a quasi-projective group, 
if $\overline{X}$ is actually a projective variety. 
A natural question arises:  which $3$-manifold 
groups are quasi-K\"{a}hler?  

In \cite{DPS07}, we gave a complete answer to this question, 
within a restricted class of $3$-manifolds: those manifolds 
$M=M(\A)$ which occur as the boundary of a regular 
neighborhood of an arrangement $\A$ of lines in $\CP^2$.  
Using an explicit formula for the Alexander polynomial of 
$M(\A)$ from \cite{CS08}, we showed that $G=\pi_1(M(\A))$ is 
quasi-K\"{a}hler if and only if $\A$ is either a pencil of $n+1$ lines, 
in which case $M(\A)=\#^n S^1\times S^2$, or a pencil of 
$n$ lines together with an extra line in general position, 
in which case $M(\A)=S^1\times \Sigma_{n-1}$, where 
$\Sigma_g$ is the closed, orientable surface of genus $g$. 

Yet there are many more $3$-manifold groups which are 
quasi-K\"{a}hler.  For instance, if $M$ is the link of 
an isolated surface singularity with $\C^*$-action, then 
its fundamental group is quasi-projective, yet not necessarily 
of the above form.  A concrete example is provided by the 
Heisenberg nilmanifold, which occurs as the Brieskorn manifold  
$\Sigma(2,3,6)$. 

\subsection{Formality}
\label{subsec:intro3}
An important topological property of compact K\"{a}hler 
manifolds is their formality. 
Following D.~Sullivan \cite{Su77}, we say that a connected, 
finite-type CW-complex $X$ is {\em formal}\/ if its minimal 
model is quasi-isomorphic to $(H^*(X,\Q),0)$.  Examples 
of formal spaces include rational homology spheres, rational 
cohomology tori, surfaces, compact Lie groups, and complements 
of complex hyperplane arrangements.  On the other hand, 
the only nilmanifolds which are formal are tori. Formality 
is preserved under wedges and products of spaces, and 
connected sums of manifolds.  

A finitely generated group $G$ is said to be {\em $1$-formal}\/ if its 
Malcev completion---in the sense of D.~Quillen \cite{Qu}---is a 
quadratic complete Lie algebra.  If $X$ is formal, then $\pi_1(X)$ is 
$1$-formal, but in general the converse does not hold. 

As shown by Deligne, Griffiths,  Morgan, and Sullivan \cite{DGMS}, 
every compact K\"{a}hler manifold is formal. Thus, all 
K\"{a}hler groups are $1$-formal and quasi-K\"{a}hler.  
Yet the converse is far from being true.  
Indeed, work of Morgan \cite{M} and Kohno \cite{Ko} 
shows that fundamental groups of complements of projective 
hypersurfaces are $1$-formal.   
For example, let $F_n$ be the free group of rank $n>0$;  
then $F_n=\pi_1(\CP^1\setminus \{\text{$n+1$ points}\})$ 
is a $1$-formal, quasi-projective group, but definitely not a 
K\"{a}hler group.

\subsection{Cohomology jumping loci}
\label{subsec:intro4}
A crucial tool for us are the cohomology jumping loci 
associated to a finite-type CW-complex $X$.  The 
{\em characteristic varieties} $V^i_d(X)$ are 
algebraic subvarieties of  $\Hom (\pi_1(X), \C^*)$, 
while the {\em resonance varieties} $R^i_d(X)$ 
are homogeneous subvarieties of  $H^1(X,\C)$.  
The former are the jump loci for the cohomology of 
$X$ with rank one twisted complex coefficients, 
while the latter are the jump loci for the homology of the cochain 
complexes arising from multiplication by degree $1$ classes in 
$H^*(X,\C)$.  The jumping loci of a group are defined in terms 
of the jumping loci of the corresponding classifying space. 

In general, the characteristic varieties are rather complicated objects. 
If the group $G=\pi_1(X)$ is $1$-formal, though, we showed in 
\cite{DPS-jump} that the analytic germ of $V^1_d(X)$ at $1$ 
coincides with the analytic germ of $R^1_d(X)$ at $0$, and  
thus may be recovered from the cup-product map 
$\cup_X\colon H^1(X,\C)\wedge H^1(X,\C) \to H^2(X,\C)$. 

Foundational results on the structure of the cohomology 
support loci for local systems on K\"{a}hler and quasi-K\"{a}hler 
manifolds were obtained by Beauville \cite{Be}, Green--Lazarsfeld 
\cite{GL}, Simpson \cite{Sp92}, and  Arapura~\cite{Ar}.  
In particular,  if $X$ is a quasi-K\"{a}hler 
manifold, then each variety $V^1_1(X,\C)$ is a union of 
(possibly translated) subtori of $\Hom (G ,\C^*)$. 

\subsection{Classification up to Malcev completion}
\label{subsec:intro5}

Our first main result gives the classification of quasi-K\"{a}hler, 
$1$-formal, $3$-manifold groups, at the level of Malcev completions.

\begin{theorem}
\label{thm:malclass}
Let $G$ be the fundamental group of a compact, connected, orientable
$3$-manifold. Assume $G$ is $1$-formal. Then the following are equivalent.
\begin{enumerate}
\item \label{m1}
The Malcev completion of $G$ is isomorphic to the Malcev completion
of a quasi-K\"{a}hler group.

\item \label{m2}
The Malcev completion of $G$ is isomorphic to the Malcev completion
of the fundamental group of $S^3$, $\# ^n S^1\times S^2$, or 
$S^1\times \Sigma_g$, with $g\ge 1$.
\end{enumerate}
\end{theorem}

Thus, when viewed through the prism of their Malcev Lie 
algebras, all quasi-K\"{a}hler, $1$-formal, $3$-manifold groups look 
like either a trivial group, a free group, or an infinite cyclic group 
times a surface group. 

The proof relies heavily on the structure of the resonance 
varieties of quasi-K\"{a}hler, $1$-formal groups, as described 
in \cite{DPS-jump}, and the resonance varieties of 
$3$-manifold groups, as described in \cite{DS07}.  An important 
role is played by the $0$-isotropic and $1$-isotropic subspaces 
of $H^1(G,\C)$, and their position with respect to the resonance 
variety $R^1_1(G)$. 

\subsection{Corank and isotropy index}
\label{subsec:intro6}
For a connected CW-complex $X$ with finite $1$-skeleton, the 
{\em isotropy index}\/ $i(X)$ is the maximum dimension of those 
subspaces $W\subseteq  H^1(X, \C)$ for which the restriction of 
$\cup_X$ to $W\wedge W$ is zero. When $X$ is a smooth complex 
algebraic variety, these subspaces $W$ are precisely the maximal 
isotropic subspaces considered by Catanese \cite{Ca} 
and Bauer \cite{Ba}.

For a finitely generated group $G$, put $i(G)=i(K(G, 1))$; it is 
readily seen that $i(X)=i(\pi_1(X))$.  It turns out that the isotropy 
index $i(G)$ depends only on $\m(G)$, the Malcev completion 
of the group. 

The {\em corank}\/ of $G$ is the 
largest integer $n$ such that there is an epimorphism 
$G\surj F_n$.  This is a subtler invariant, depending on 
more information, even for $1$-formal groups.  For example, 
if $G=\pi_1(\#^3 \RP^2)$, then $G$ is $1$-formal, $\m(G)=\m(F_2)$, 
yet $\corank(G)<\corank(F_2)$. Restricting to the class of 
($1$-formal) quasi-K\"{a}hler groups, we show that the corank 
and the isotropy index coincide; thus, such a phenomenon 
cannot happen within this class. The result below improves Lemma 6.16
from \cite{DPS-jump}.

\begin{theorem}
\label{thm:cutiso}
Let $G$ be a quasi-K\"{a}hler, $1$-formal group.  Then
$\corank(G)=i(G)$. In particular, the corank of $G$ depends  
only on its Malcev completion. 
\end{theorem}

The $1$-formality assumption is crucial here. For example, 
let $G$ be the fundamental group of the Heisenberg 
nilmanifold. As noted  in  \S\ref{subsec:intro2} (see also 
Morgan \cite[p.~203]{M}), $G$ is a quasi-K\"{a}hler group.  
Nevertheless, $\corank(G)=1$, yet $i(G)=2$.  

Suppose now our quasi-K\"{a}hler, $1$-formal group  
is also a $3$-manifold group. 
Using Theorems \ref{thm:malclass} and \ref{thm:cutiso}, 
we can compute its corank explicitly.

\begin{corollary}
\label{cor:cutiso}
Let $G$ be a quasi-K\"{a}hler, $1$-formal, $3$-manifold group.  
Then the Malcev completion of $G$ is that of either $1$, $F_n$, 
or $\Z\times \pi_1(\Sigma_g)$, in which case $\corank(G)$ 
equals $0$, $n$, or $g$. In particular, if $b_1(G)=2k$, then $\m(G)=\m(F_{2k})$, 
and $\corank(G)=b_1(G)$.  
\end{corollary}

\subsection{Quasi-homogeneous surface singularities}
\label{subsec:intro7}

Up to this point, we have only discussed quasi-K\"{a}hler, 
$3$-manifold groups which are $1$-formal.  But there are 
many groups of this sort which are not $1$-formal.   
A rich supply of such groups occurs in the context 
of isolated surface singularities with $\C^*$-action. 

Let $(X,0)$ be such a singularity, and let $M$ be the associated 
singularity link (a closed, oriented $3$-manifold).  Then clearly 
$G=\pi_1(M)$ is a quasi-projective, $3$-manifold group.  
Yet we show in Proposition \ref{prop:qh} that $G$ is not 
$1$-formal, provided $b_1(M)>0$.  

In the same vein, we construct in Proposition \ref{prop:affine} 
irreducible, smooth {\em affine}\/  surfaces $U$, with 
non-$1$-formal fundamental group $G$. To the best of 
our knowledge, these are the first examples of this kind. 
Note that the Deligne weight filtration subspace $W_1H^1(U, \C)$
must be non-trivial, since otherwise $G$ would be $1$-formal, 
according to Morgan's results from \cite{M}.

Computing the characteristic varieties of groups which 
are not $1$-formal can be a very arduous task.  Nevertheless, 
this can be done in our context:  in Proposition \ref{propN}, 
we give a precise description of all positive-dimensional 
irreducible components of $V^1_1(M)$, in the case when 
$M$ is the link of an isolated surface singularity with 
$\C^*$-action. In particular, we describe the analytic 
germs at the origin of those components, solely in 
terms of $b_1(M)$. Detailed computations are carried 
out in the case when $M=\Sigma(a_1,\dots,a_n)$ is 
a Brieskorn manifold. 

\subsection{Organization of the paper}
\label{intro:org}

In Section \ref{sect:malcev} we review some basic 
material on Malcev completions and $1$-formality, 
while in Section \ref{sect:cjl} we review cohomology 
jumping loci and the tangent cone theorem for $1$-formal groups. 
In Section \ref{sect:isotropic}, we discuss the characteristic 
and resonance varieties of quasi-K\"{a}hler manifolds and 
closed $3$-manifolds, with particular attention to their 
isotropicity properties. 

Section \ref{sect:qk3} is devoted to a proof of 
Theorem \ref{thm:malclass}, based on the different 
qualitative properties of resonance varieties in the 
quasi-K\"{a}hler and $3$-dimensional settings.  
We start Section \ref{sect:isocut} with a discussion 
of the corank and isotropy index; the rest of the 
section is devoted to proving Theorem \ref{thm:cutiso} 
and Corollary \ref{cor:cutiso}, and to an application. 

In Section \ref{sect:surf}, we study the formality 
properties of links of quasi-homogeneous surface singularities. 
And finally, in Section \ref{sect:trans}, we discuss the 
translated components in the characteristic varieties 
of these singularity links.  

\section{Malcev completion and $1$-formality}
\label{sect:malcev}

\subsection{Malcev completion}
\label{subsec:malcev lie}

Let $G$ be a group. The lower central series 
of $G$  is defined inductively by $\gamma_1 G=G$ 
and $\gamma_{k+1}G =[\gamma_k G,G]$, where 
$[x,y]=xyx^{-1}y^{-1}$. The {\em associated 
graded Lie algebra}, $\gr(G)$, is the direct sum 
of the successive LCS quotients, 
$\gr(G)= \bigoplus\nolimits_{k\ge 1} \gamma_k G/ \gamma_{k+1} G$, 
with Lie bracket induced from the group commutator; 
see \cite{MKS}. 

In \cite[Appendix A]{Qu}, Quillen associates to any group $G$ 
a pronilpotent, rational Lie algebra, $\m(G)$, called 
the {\em Malcev completion} of $G$.  The Malcev Lie algebra 
comes endowed with a decreasing, complete, $\Q$-vector 
space filtration, such that the associated graded Lie algebra, 
$\gr(\m(G))$, is isomorphic to $\gr(G)\otimes \Q$.  For precise 
definitions and basic properties of Malcev Lie algebras, we refer 
to \cite{PS-chenlie}. 

Now suppose $G$ is a finitely generated group.  We then 
have the following result of Sullivan \cite{Su77}. 

\begin{lemma}
\label{lem:malcev}
The Malcev completion $\m(G)$ determines the 
corestriction to its image of the cup-product map 
$\cup_G \colon H^1(G, \Q)\bigwedge H^1(G, \Q)\to H^2(G, \Q)$. 
\end{lemma}

In particular, the first Betti number, $b_1(G)= \dim H^1(G, \Q)$, 
is determined by $\m(G)$. Note also that if $G=\pi_1(X)$, then 
$\cup_G$ and $\cup_X$ have the same corestriction.

We will also need the following fact, extracted from \cite{Su77}. 

\begin{lemma}
\label{lem:sullivan}
Let $\phi\colon G'\to G$ be a group homomorphism. If
$\phi^1\colon H^1(G, \Q)\to H^1(G', \Q)$ is an isomorphism, and 
$\phi^2\colon H^2(G, \Q)\to H^2(G', \Q)$ is a monomorphism, then 
$\phi$ induces an isomorphism $\phi_*\colon \m(G')\to \m(G)$ 
between Malcev completions. 
\end{lemma}

Suppose now $G=\pi_1(X)$ and $G'=\pi_1(X')$. Let $f\colon X'\to X$ be a 
continuous map, and set $\phi= f_{\sharp}\colon G'\to G$. Recall that a 
classifying map, $\kappa\colon X\to K(G, 1)$, induces an isomorphism 
on $H^1$ and a monomorphism on $H^2$. Hence, $\phi$ 
satisfies the cohomological hypotheses of Lemma \ref{lem:sullivan} 
whenever $f^*\colon H^*(X, \Q)\to H^*(X', \Q)$ satisfies them.

\subsection{$1$-formality}
\label{subsec:1formal}
As before, let $X$ be a connected CW-complex, with finite 
$1$-skeleton.  Following  K.-T. Chen \cite{Ch}, define the
{\em (rational) holonomy Lie algebra} of $X$ as 
\begin{equation}
\label{eq:hlie}
\h (X) = \Lie(H_1(X, \Q))/\, \text{ideal} (\im (\partial_X)),
\end{equation}
where $\Lie(H_1(X, \Q))$ is the free (graded) Lie algebra over $\Q$, 
generated in degree $1$ by $H_1(X; \Q)$, and $\partial_X\colon 
H_2(X, \Q)\to  H_1(X, \Q) \wedge H_1(X, \Q)$ is
the dual of the cup-product map, $\cup_X$.  Since 
its defining ideal is homogeneous (in fact, quadratic), 
$\h (X)$ inherits a natural grading from the free Lie algebra.   
Note that $\h(X)$ depends only on the corestriction 
of $\cup_X $ to its image.  It follows that $\h(X) =\h(\pi_1(X))$, 
where $\h(G):=\h(K(G,1))$. 

Now let $G$ be a finitely generated group, and 
let  $\wh(G)$ be the completion of $\h(G)$ with respect 
to the degree filtration. Following Sullivan \cite{Su77}, 
we say $G$ is {\em $1$-formal}\/ if $\m(G)\cong \wh(G)$, 
as (complete) filtered Lie algebras.  Equivalently, $\m(G)$ 
is filtered Lie isomorphic to the degree completion of a 
quadratic Lie algebra.  

It follows from the definitions that the Malcev completion 
of a $1$-formal group $G$ is determined by the corestriction 
of $\cup_G $ to its image; or, the corestriction of $\cup_X$ 
to its image, if $G=\pi_1(X)$, where $X$ is a CW-complex 
with finite $1$-skeleton. 

\begin{example}
\label{ex:free}
Let $F_n$ be the free group of rank $n$.  Then, as shown 
by Witt and Magnus, $\gr(F_n)\otimes \Q=\LL_n$, the free 
Lie algebra of rank $n$.  It is readily checked that $\h(F_n)=\LL_n$, 
and $\m(F_n)=\widehat{\LL}_n$.  In particular, $F_n$  
is $1$-formal. 
\end{example}

\begin{example}
\label{ex:surface}
Let $\Sigma_g$ be the Riemann surface of genus $g\ge 1$, 
with fundamental group $\Pi_g=\pi_1(\Sigma_g)$ generated 
by $x_1,y_1,\dots,x_g,y_g$, subject to the relation 
$[x_1,y_1]\cdots [x_g,y_g]=1$. It follows from \cite{DGMS} that 
the group $\Pi_g$  is $1$-formal. Using \eqref{eq:hlie}, 
it is readily checked that $\h(\Pi_g)$ 
is the quotient of the free Lie algebra on $x_1,y_1,\dots,x_g,y_g$  
by the ideal generated by $[x_1,y_1]+ \cdots + [x_g,y_g]$. 
\end{example}

\section{Cohomology jumping loci}
\label{sect:cjl}

\subsection{Characteristic varieties}
\label{subs:cvs}
Let $X$ be a connected CW-complex with finitely many cells 
in each dimension. Let $G=\pi_1(X)$ be the fundamental group 
of $X$, and $\T(G) =\Hom (G, \C^*)$ its character variety. 
Every character $\rho\in \T(G)$ determines a rank $1$ local 
system, $\C_{\rho}$, on $X$.   The {\em characteristic varieties}\/ 
of $X$ are the jumping loci for cohomology with coefficients in 
such local systems:
\begin{equation}
\label{eq:cv}
V^i_d(X)=\{\rho \in \T(G)
\mid \dim H^i(X,\C_{\rho})\ge d\}. 
\end{equation}

The varieties $V_d(X)=V^1_d(X)$ depend only on 
$G=\pi_1(X)$, so we sometimes denote them as $V_d(G)$.
Here is an alternate description, along the lines of \cite{DPS07}.  

Let $1$ be the identity of the algebraic group 
$\T(G)$, and let $\T^0(G)$ be the connected component 
containing $1$.  Set $n=b_1(G)$, and let $G_{\abf}=\Z^n$ 
be the maximal torsion-free abelian quotient of $G$. 
The group algebra $\C{G}_{\abf}$ may be identified with 
$\Lambda_n=\C [t_1^{\pm 1},\dots ,t_n^{\pm 1}]$, 
the coordinate ring of the algebraic torus $\T^0(G)=(\C^*)^n$. 
Finally, set $A_G=\C{G_{\abf}} \otimes_{\C{G}} I_G$,  
where $\C{G}$ is the group algebra of $G$, with augmentation 
ideal $I_G$.  Then, for all $d< n$ (and away from $1$ if $d\ge n$), 
\begin{equation}
\label{eq:vdg}
V_d(G) \cap \T^0(G)=V(E_d (A_G)),
\end{equation}
the subvariety of $\T^0(G)$ defined by $E_d (A_G)$, the 
ideal of codimension $d$ minors of a presentation matrix 
for $A_G$;  see \cite[Proposition 2.4]{DPS07}.  Such a 
matrix can be computed explicitly from a finite presentation 
for $G$, by means of the Fox calculus.  

In favorable situations, the intersection of $V_1(G)$  with the 
torus $ \T^0(G)$ is a hypersurface, defined by the vanishing 
of the {\em  Alexander polynomial}, $\Delta^G$.   More precisely, 
let $\Delta^G\in \Lambda_n$  be the greatest common divisor 
of all elements of $E_1(A_G)$, up to units in $\Lambda_n$. 
Suppose the Alexander ideal $E_1(A_G)$ is {\em almost principal}, 
i.e., there is an integer $k>0$ such that 
$I_G^k \cdot (\Delta^G) \subset E_1(A_G)$.  
We then have $V_1(G) \cap \T^0(G)=V(\Delta^G)$, at least away 
from $1$.  For more details, see \cite[\S6.1]{DPS07}.

\subsection{Resonance varieties}
\label{subs:rvs}
Consider now the cohomology algebra $H^* (X,\C)$.  
Left-mul\-tiplication by an element $x\in H=H^1(X,\C)$ 
yields a cochain complex $(H^*(X,\C), \lambda_x )$.  
The {\em resonance varieties}\/ of $X$ are the jumping 
loci for the homology of this complex:
\begin{equation}
\label{eq:rv}
R^i_d(X)=\{x \in H \mid 
\dim H^i(H^*(X,\C),\lambda_x ) \ge  d\}.
\end{equation}

The homogeneous varieties $R_d(X)=R^1_d(X)$ depend 
only on $G=\pi_1(X)$, so we sometimes denote them by 
$R_d(G)$.  In view of Lemma \ref{lem:malcev}, the varieties 
$R_d(G)$ depend only on the Malcev completion, $\m(G)$. 

\subsection{The tangent cone theorem}
\label{subs:tct}
Identify $H^1(X,\C)$ with $T_1 \T(G)= \Hom (G, \C)$,  
the Lie algebra of the algebraic group $\T(G)$, 
and let $\exp \colon T_1 \T(G) \to \T(G)$ be the
exponential map. Under this identification, 
the tangent cone to $V^i_d(X)$ at the origin, 
$TC_{1}(V^i_d(X))$, is contained in $R^i_d(X)$, 
see Libgober \cite{Li}.   In general, though, this 
inclusion is strict. 

The main link between the $1$-formality property of a group 
and its cohomology jumping loci is provided by the following 
theorem. 

\begin{theorem}[Dimca--Papadima--Suciu \cite{DPS-jump}, Theorem A] 
\label{thm:tcone}
Suppose $G$ is a $1$-formal group.  Then, for each $d\ge 1$, 
the exponential map induces a complex analytic isomorphism 
between the germ at $0$ of $R_d(G)$ and the germ at $1$ 
of $V_d(G)$. Consequently,  $TC_{1}(V_d(G))=R_d(G)$. 
\end{theorem}

\section{Isotropicity properties}
\label{sect:isotropic}

\subsection{Isotropic subspaces}
\label{subsec:isotropy}
Let $G$ be a finitely generated group. 
We say that a non-zero subspace $W\subseteq  H^1(G,\C)$ 
is {\em $p$-isotropic} ($p=0$ or $1$) with respect to the cup-product 
map $\cup_G\colon  \bigwedge^2 H^1(G,\C)\to H^2(G,\C)$ if the 
restriction of $\cup_G$ to $\bigwedge^2 W$ has rank $p$.  

The motivation for this definition comes from the case when 
$G=\pi_1(C)$ is the fundamental group of a non-simply-connected 
smooth complex curve $C$.  Set $W= H^1(G,\C)$. 
If $C$ is non-compact, then 
$G=F_n$, for some $n>0$, and so $W$ is $0$-isotropic.  
On the other hand, if  $C$ is compact, then 
$G=\Pi_g$, for some $g>0$, and so $W$ is $1$-isotropic. 

Note that the $p$-isotropicity property depends only on the 
Malcev completion, by Lemma \ref{lem:malcev}.

\subsection{Characteristic varieties of quasi-K\"{a}hler manifolds}
\label{subsec:cvqk}
Let $X$ be a quasi-K\"{a}hler manifold, with fundamental 
group $G=\pi_1(X)$.  As mentioned in \S\ref{subsec:intro4}, 
it is known that $V_1(G)$ is a union of (possibly translated) 
subtori of $\T(G)=\Hom (G ,\C^*)$. 
Let us describe in more detail those components passing 
through the identity $1\in \T(G)$. 

Following Arapura \cite[p.~590]{Ar}, we say that a map
$f\colon X \to C$ to a connected, smooth complex curve 
$C$ is {\em admissible}\/ if $f$ is holomorphic and surjective, 
and has a holomorphic, surjective extension with connected 
fibers to smooth compactifications, 
$\overline{f}\colon \overline{X} \to \overline{C}$, 
obtained by adding divisors with normal crossings. 
In particular, the generic fiber of $f$ is connected, and the induced 
homomorphism, $f_{\sharp}\colon \pi_1(X)\to \pi_1(C)$, is onto. 
Two such maps, $f\colon X \to C$ and $f'\colon X\to C'$, 
are said to be equivalent if there is an isomorphism 
$\psi\colon C \to C'$ such that $f'=\psi\circ f$.  
Proposition V.1.7 from \cite{Ar} can now be stated, as follows.
  
\begin{theorem}[Arapura \cite{Ar}]
\label{thm:arapura}
There is a bijection between the set of positive-dimensional irreducible 
components of $V_1(G)$ containing $1$, and the set of equivalence 
classes of admissible maps $f\colon X\to C$ with $\chi(C)<0$. 
This bijection associates to $f$ the subtorus 
$S_f= f^*\T(\pi_1(C))\subseteq \T^0(G)$. 
\end{theorem}

\subsection{Resonance varieties of quasi-K\"{a}hler manifolds}
\label{subsec:res qk}
As above, let $X$ be a quasi-K\"{a}hler manifold. 
Let us assume now that $G=\pi_1(X)$ is $1$-formal.  
Using the aforementioned results of Arapura, in 
conjunction with Theorem \ref{thm:tcone}, we establish 
in \cite[Theorem C]{DPS-jump} a bijection between 
the positive-dimensional irreducible components of 
$R_1(G)$ and admissible maps $f\colon X\to C$ 
with $\chi(C)<0$.  Under this bijection, an 
admissible map $f$ corresponds to the linear subspace 
$W_f= f^*H^1(C, \C)=T_1(S_f)\subseteq H^1(G, \C)$.
As explained in detail in \cite{Di08}, this correspondence 
is closely related to the results of Catanese \cite{Ca} 
and Bauer \cite{Ba} on isotropic subspaces. 

With this notation, Proposition 7.2(3) from \cite{DPS-jump} 
can be reformulated in a more precise fashion, as follows.  

\begin{theorem}[Dimca--Papadima--Suciu \cite{DPS-jump}]
\label{thm:posobstr} 
Let $X$ be  a quasi-K\"{a}hler manifold, with $1$-formal fundamental 
group $G$.  Suppose $f\colon X\to C$ is an admissible map onto 
a smooth curve $C$ with $\chi(C)\le 0$.  There are then % exactly 
two possibilities.
\begin{enumerate}
\item  \label{a1}  
$W_f$ is a $0$-isotropic subspace, which happens 
precisely when $C$ is non-compact. 

\item  \label{a2} 
$W_f$ is a $1$-isotropic subspace, which happens 
precisely when $C$ is compact.
\end{enumerate}
\end{theorem}

\subsection{Jumping loci of $3$-manifolds}
\label{subsec:res3}

Let $M$ be a compact, connected, orientable $3$-manifold, 
with fundamental group $\pi$, and Alexander polynomial 
$\Delta^{\pi}$.  As shown by McMullen in \cite[Theorem 5.1]{McM}, 
$I^2_{\pi}\cdot (\Delta^{\pi})\subset E_1(A_{\pi})$, and so 
the Alexander ideal $E_1(A_{\pi})$ is almost principal. 
From the discussion at the end of \S\ref{subs:cvs}, we have  
\begin{equation}
\label{eq:v1 3m}
V_1(M)\cap \T^0(\pi)= V(\Delta^{\pi}), 
\end{equation}
at least away from $1$. 

Now fix an orientation on $M$; that is to say, pick a generator 
$[M]\in H_3(M,\Z)\cong\Z$. With this choice, the cup 
product on $M$ and Kronecker pairing determine an 
alternating $3$-form $\mu_M$ on $H^1(M,\Z)$, given by 
$\mu_M(x,y,z) = \langle x\cup y\cup z , [M]\rangle$. 

In coordinates, we can write this form as follows. 
Choose a basis $\{e_1,\dots ,e_n\}$ for $H_1(M,\C)$. 
Let $\{e^*_1,\dots ,e^*_n\}$ be the Kronecker dual basis 
for $H^1(M,\C)$, and let $\{e^{\vee}_1,\dots ,e^{\vee}_n\}$ be 
the Poincar\'{e} dual basis for $H^2(M,\C)$, defined by 
$\langle e^{\vee}_i \cup e^*_j, [M] \rangle= \delta_{ij}$. 
Set $\mu_{ijk}= \langle e^*_i \cup e^*_j \cup e^*_k, [M] \rangle$. 
The $3$-form $\eta=\mu_M \otimes \C$ is then given by 
\begin{equation}
\label{eq:eta}
\eta= \sum_{i,j,k} \mu_{ijk} e_i \wedge e_j \wedge e_k.
\end{equation}

\begin{remark}
\label{rem:sullivan}
In \cite{Su75}, Sullivan showed that {\em any}\/ alternating 
$3$-form $\eta\in \bigwedge^3 (\Q^n)^*$ can be realized as 
$\eta=\mu_M \otimes \Q$, for some closed, connected, 
oriented $3$-manifold $M$ with $b_1(M)=n$.  He also 
showed that, if $X$ is a complex algebraic surface with 
an isolated singularity, and $M$ is the link of that singularity, 
then $\mu_M=0$. 
\end{remark}

We will need the following result from \cite{DS07}, 
which provides information on the first 
resonance variety of a $3$-manifold, and its 
isotropicity properties.  

\begin{theorem}[Dimca--Suciu \cite{DS07}]  %Proposition 5.1
\label{thm:res3}
Let $M$ be a closed, orientable $3$-manifold.  
Then:
\begin{enumerate}
\item  \label{r1}
$H^1(M,\C)$ is not $1$-isotropic. 
\item \label{r2}
If $b_1(M)$ is even, then $R_1(M)=H^1(M,\C)$.
\item \label{r3}
If $R_1(M)$ contains a $0$-isotropic hyperplane, then 
$R_1(M)= H^1(M, \C)$.
\end{enumerate}
\end{theorem}

In the next remark, we provide an alternative proof of Part \eqref{r2} 
of this theorem, and further information on the remaining case, 
when $b_1(M)$ is odd. 

\begin{remark}
\label{rem:bp}
Assume $n=b_1(M)>0$, and let $x\in H^1(M,\C)$ be a non-zero 
element. By definition, $x \not\in R_1(M)$ if and only if the linear map
$\lambda_x\colon H^1(M, \C) \to H^2(M, \C)$ has rank $n-1$. 
Choose a basis $\{e_1,\dots ,e_n\}$ for $H_1(M,\C)$ so that 
$x=e_n^*$. With notation as above, identify the subspaces spanned by 
$\{e^*_1,\dots ,e^*_{n-1} \}$ and $\{e^{\vee}_1,\dots ,e^{\vee}_{n-1} \}$
with $\C^{n-1}$, and let $\lambda\colon \C^{n-1} \to \C^{n-1}$ 
be the  restriction of $\lambda_x$ to those subspaces. 
The matrix of $\lambda$ is skew-symmetric, with entries 
$\mu_{ijn}$, for $1\le i,j <n$; moreover, $\lambda$ and 
$\lambda_x$ have the same rank. Hence, $x \not\in R_1(M)$ 
if and only if $\det (\lambda)\ne 0$.  There are two cases to consider. 

If $n=b_1(M)$ is even, it follows that $\det (\lambda)=0$. 
Therefore, $R_1(M)=H^1(M,\C)$.

If $n=b_1(M)$ is odd, say $n=2m+1$, it follows that 
$R_1(M)\ne H^1(M,\C)$ if and only if
\begin{equation}
\label{eq:generic}
\eta = (e_1 \wedge e_2 + \dots + e_{2m-1} \wedge 
e_{2m})\wedge e_{2m+1}+\sum_{1\le i,j,k\le 2m} \mu_{ijk} 
e_i \wedge e_j \wedge e_k  ,
\end{equation}
with respect to some basis $\{e_1,\dots ,e_n\}$. This means 
precisely that the Poincar\'{e} duality algebra $H^*(M, \C)$
is generic, in the sense of definition $(3')$ from \cite[p.~463]{BP}. 
\end{remark}

\section{Quasi-K\"{a}hler $3$-manifold groups}
\label{sect:qk3}

This section is devoted to the proof of Theorem \ref{thm:malclass} 
from the Introduction, namely, the nontrivial 
implication, \eqref{m1}$\Rightarrow$\eqref{m2}. 
Recall $\Sigma_g$ is the Riemann surface 
of genus $g>0$, with fundamental group $\Pi_g$. 
It is enough to prove the following lemma. 

\begin{lemma}
\label{lem:mal3k}
Let $M$ be a closed, orientable $3$-manifold, and  
let $X$ be a quasi-K\"{a}hler manifold. Assume both 
$\pi=\pi_1(M)$ and $G=\pi_1(X)$ are $1$-formal, 
and their Malcev completions are isomorphic. 
Then $\m(\pi)$ is isomorphic to either $0$, 
$\m(F_n)$, or $\m(\Z\times \Pi_g)$.  
\end{lemma}

\begin{proof}
The proof is broken into four steps.

\subsubsection*{Step 1}
\label{step1}
Recall from \S\ref{subsec:malcev lie} that the Malcev completion  
of a group determines its first Betti number.  Since, by assumption,  
$\m(\pi)\cong \m(G)$, we must have $b_1(\pi)=b_1(G)$.  Let $n$ denote 
this common Betti number.  We first dispose of some easy cases. 

Suppose $\mu_M=0$.  Then $\cup_M=\cup_{M_n}$, where 
$M_0=S^3$ and $M_n=\#^n S^1\times S^2$ for $n>0$.  Since 
$\pi_1(M_n)=F_n$ is $1$-formal, $\m(\pi)=\m(F_n)$, and 
we are done.  So we may assume for the rest of the proof 
that  $\mu_M\ne 0$.  This immediately forces $n\ge 2$.  

Now suppose $n=3$.  Since $\mu_M\ne 0$, we must have 
$\mu_M\otimes \Q=\mu_{T^3}\otimes \Q$, where 
$T^3=S^1\times \Sigma_1$ is the $3$-torus.  
Since $\pi_1(T^3)=\Z^3$ is $1$-formal, $\m(\pi)=\m(\Z^3)$, 
and we are done.  Thus, for the rest of the proof, we may 
assume $n\ge 4$. 

\subsubsection*{Step 2}
\label{step2}
This step takes care of the case when $R_1(X)= H^1(X, \C)$. 
From the discussion in \S\ref{subsec:res qk}, we know 
there is an admissible map $f\colon X\to C$ such that 
$W_f=H^1(X, \C)$.  According to Theorem \ref{thm:posobstr}, 
the subspace $W_f$ is either $0$-isotropic or $1$-isotropic. 

On the other hand, since $\m (\pi)=\m (G)$, we have 
$H^1(X, \C)= H^1(M, \C)$.  In view of Theorem \ref{thm:res3}\eqref{r1}, 
the space $H^1(M, \C)$ must be $0$-isotropic. Therefore, 
$W_f$ is $0$-isotropic.  Hence, by Theorem \ref{thm:posobstr}, 
the curve $C$ is non-compact. 

By construction, the map $f^1\colon H^1(C, \Q)\to H^1(X, \Q)$ 
is an isomorphism. Since $C$ is non-compact, 
$f^2\colon H^2(C, \Q)\to H^2(X, \Q)$ is trivially a monomorphism. 
Applying now Lemma \ref{lem:sullivan}, we conclude that $G$ 
has the same Malcev completion as $\pi_1(C)=F_n$.

\subsubsection*{Step 3}
\label{step3}
We now may assume that $R_1(X)\ne H^1(X, \C)$, which 
implies $R_1(M)\ne H^1(M, \C)$.  From 
Theorem \ref{thm:res3}\eqref{r2}, we infer that $n=2g+1$, 
for some $g\ge 2$. We claim that, in this case, 
$R_1(X)$ is a hyperplane, defined over $\Q$.

Indeed, Theorem \ref{thm:tcone} guarantees that 
$R_1(X)=TC_1(V_1(X))$, and likewise for $M$.
On the other hand, we know from \eqref{eq:v1 3m} that
$V_1(M)\cap \T^0(\pi)= V(\Delta^{\pi})$, away from $1$. 
Since $n\ge 4$, both subvarieties of $\T^0(\pi)$ contain $1$; 
see \cite[Proposition 3.11]{MP}. In particular, $\Delta^{\pi}\ne 0,1$. 
Hence, all irreducible components of $R_1(M)$ have codimension $1$.

The same property holds for $R_1(X)$. By the discussion preceding 
Theorem \ref{thm:posobstr}, all irreducible components of $R_1(X)$ 
must be hyperplanes, defined over $\Q$. Using Theorem 4.2 from 
\cite{DPS07} (which continues to hold in the quasi-K\"{a}hler case),
we infer that any two distinct components of $R_1(X)$ intersect only 
at $0$. Since $n\ge 3$, the variety $R_1(X)$ must be irreducible, 
and this proves our claim.

\subsubsection*{Step 4}
\label{step4}
By Theorems \ref{thm:posobstr} and \ref{thm:res3}\eqref{r3}, 
$R_1(X)=f^*H^1(C, \Q)$, the curve $C$ is compact of genus $g$, and 
$\omega:= f^*(\omega_0)\ne 0$, where $\omega_0\in H^2(C, \Q)$ is 
the generator corresponding to the complex orientation.

Pick a map $h\colon X\to S^1$ with the property that 
$t:= h^*(\eta_0)\not\in R_1(X)$, where $\eta_0\in H^1(S^1, \Q)$ 
is the orientation generator. Set $F=(h,f)\colon X\to S^1\times C$. 
By construction, $F^1\colon H^1(S^1\times C,\Q)\to H^1(X,\Q)$ 
is an isomorphism. We claim that $F^2\colon 
H^2(S^1\times C,\Q)\to H^2(X,\Q)$ is a monomorphism. 

To verify the claim, decompose the source according to 
the K\"{u}nneth formula, 
\[
H^2(S^1\times C, \Q)= 
\Q \cdot \omega_0 \oplus \big ( \Q \cdot \eta_0 \otimes H^1(C, \Q)\big ),
\]
and consider the restriction of $F^2$ to each of the two summands.

Firstly, $F^2$ is injective on the second summand. For otherwise, 
we could find an non-zero element $z\in R_1(X)$ with the property 
that $t\cup z=0$. But this contradicts our assumption that 
$t\not\in R_1(X)$. 

Secondly, we must check that $\omega \not\in t\cup R_1(X)$. 
Supposing the contrary, let $\omega =t\cup x$, with $0\ne x\in R_1(X)$. 
By Poincar\'{e} duality on $C$, we may find $y\in R_1(X)$ 
such that $\omega =y\cup x$. Thus, $(t-y)\cup x=0$, and so  
$t-y\in R_1(X)$, contradicting again the choice of $t$. 

The claim is now proved. Using Lemma \ref{lem:sullivan}, we conclude 
that $\m(G)=\m(\Z\times \Pi_g)$, and we are done. 
\end{proof}

\section{Corank and isotropy index}
\label{sect:isocut}

\subsection{Two numerical invariants}
\label{subsec:corank}

To every finitely generated group $G$, we 
associated in \S\ref{subsec:intro6} two non-negative 
integers---the corank and the isotropy index:
\begin{align}
\corank (G)&=\max \{n \mid \exists \: \text{epimorphism $G\surj F_n$}\},\\
i(G)&= \max\{ d \mid  \exists \: W\subseteq  H^1(G, \C),\: \dim W=d,\: 
\left.\cup_G\right|_{W\wedge W} =0\}.  \notag
\end{align}

Clearly, both these invariants are bounded above by the 
first Betti number $b_1(G)$.  If there is an epimorphism $G\surj H$, then  
$\corank(G)\ge \corank (H)$. In particular, if $b_1(G)>0$, then 
$\corank(G)>0$. Moreover, if $\phi\colon G\surj F_n$ 
is an epimorphism, the restriction of $\cup_G$ to 
$\phi^1(H^1(F_n,\C))$ vanishes; thus, 
\begin{equation}
\label{eq:coind}
\corank (G)\le i(G).
\end{equation}

The corank of a $3$-manifold group has been studied 
by Harvey \cite{Ha}, Leininger and Reid \cite{LR}, 
and Sikora \cite{Si}.  In particular, if $M$ is a closed 
$3$-manifold, the corank of $\pi_1(M)$ equals the 
``cut number" of $M$.

In view of Lemma \ref{lem:malcev}, the isotropy index $i(G)$ 
depends only on the Malcev completion $\m(G)$. In general, though, 
the corank of $G$ is not determined by $\m(G)$, even when 
$G$ is $1$-formal.  Adapting Example 6.18 from \cite{DPS-jump} 
to our situation illustrates this phenomenon.

\begin{example}
\label{ex:nonorientable}
Let $N=\#^3 \RP^2$ be the non-orientable  surface of genus 
$3$. It is readily seen that $N$ has the rational homotopy 
type of $S^1\vee S^1$.  Hence, $N$ is a formal space, 
and $G=\pi_1(N)$ is a $1$-formal group, with $b_1(G)=2$. 
Moreover, $\m(G)\cong \m(F_2)$.  

Suppose there is an epimorphism $\phi\colon G\surj F_2$. 
Then $\phi^* H^1(F_2, \Z_2)$ is a $2$-dimensional 
subspace of $H^1(G, \Z_2)=\Z_2^3$, isotropic with 
respect to $\cup_{G}$.  But this is impossible, by Poincar\'{e}  
duality with $\Z_2$ coefficients in $N=K(G,1)$.  Hence, 
$\corank (G)=1$, though of course, $\corank(F_2)=2$. 
\end{example}

\subsection{Proof of Theorem \ref{thm:cutiso}}
\label{subsec:cutiso1}
Taking into account inequality \eqref{eq:coind}, it is enough 
to prove the following lemma.

\begin{lemma}
\label{lem:corank}
Let $X$ be a quasi-K\"{a}hler manifold.  
Assume $G=\pi_1(X)$ is $1$-formal.  
Then $\corank (G)\ge i(G)$. 
\end{lemma}

\begin{proof}
Set $d:=i(G)$.  We may assume $d\ge 2$. 
Indeed, if $d=0$ we are done, and if $d=1$ then 
$b_1(G)\ge 1$, which implies $\corank (G)\ge 1$. 

Let $W\subseteq H^1(X, \C)$ be a $d$-dimensional, $0$-isotropic 
subspace. Clearly, $W\subseteq R_1(X)$, since $d\ge 2$.
By the discussion preceding Theorem \ref{thm:posobstr}, there 
is an admissible map, $f\colon X\to C$, onto a smooth complex 
curve with $\chi(C)<0$, such that $W\subseteq W_f$. 
In particular, we have an epimorphism $f_{\sharp}\colon G\surj \pi_1(C)$.  
The argument splits according to  the two cases from 
Theorem \ref{thm:posobstr}. Set $b=b_1(C)$.

\begin{enumerate}
\item {\em $C$ is non-compact}. 
In this case,  $\pi_1(C)=F_b$. Thus, we have a surjection 
$f_{\sharp}\colon G\surj F_b$, and so $\corank (G)\ge b$. 
On the other hand, $b=\dim W_f\ge d$.  

\item {\em $C$ is compact}.
In this case, $b=2g$, where $g$ is the genus of $C$, and $W$ is a 
$0$-isotropic subspace of the $1$-isotropic space $H^1(C, \C)$. 
Hence, $d\le g$. Composing the obvious epimorphism 
$\pi_1(C)\surj F_g$ with $f_{\sharp}$, we obtain a surjection 
$G\surj F_g$. Hence, $g\le \corank (G)$. 
\end{enumerate}

In either case, we conclude that $d\le \corank(G)$, and 
this ends the proof. 
\end{proof}

\subsection{Proof of Corollary \ref{cor:cutiso}}
\label{subsec:cutiso2}
We need to determine the corank of a quasi-K\"{a}hler, 
$1$-formal, $3$-manifold group.  In view of 
Theorems \ref{thm:malclass} and \ref{thm:cutiso},  
it is enough to compute $i(S^3)$, $i(\# ^n S^1\times S^2)$, 
and $i(S^1\times \Sigma_g)$. We do this next. 

\begin{lemma}
\label{lem:index}
The following hold.
\begin{enumerate}
\item \label{i1}
If $M=S^3$ or $M= \# ^n S^1\times S^2$, then $i(M)=b_1(M)$.
\item \label{i2}
If $M= S^1\times \Sigma_g$, with $g\ge 1$, then 
$i(M)=\frac{1}{2} (b_1(M)-1)$.
\end{enumerate}
\end{lemma}

\begin{proof}
If $M=S^3$ or $M=\#^n S^1\times S^2$, then $\cup_M=0$, 
and so $i(M)=b_1(M)$. If $M=S^1\times \Sigma_1$, it is 
readily seen that $i(M)=1$, as asserted.

In the remaining case,  $M=S^1\times \Sigma_g$, with $g\ge 2$, we 
have to check that $g=i(M)$. By considering a maximal 
$0$-isotropic subspace $W\subseteq H^1(\Sigma_g, \C)$, we 
see that $g\le i(M)$.  A standard argument now shows that 
every $0$-isotropic subspace $W\subseteq H^1(M, \C)$ of dimension 
$i(M)$ is contained in $H^1(\Sigma_g, \C)$; thus, $g\ge i(M)$. 
\end{proof}
 
\subsection{On a certain class of $3$-manifolds}
\label{subsec:cutiso3}
We conclude this section with an application of 
Corollary \ref{cor:cutiso} to low-dimensional topology.  
The next result identifies the class of $3$-manifolds 
whose fundamental groups satisfy certain properties. 

\begin{corollary}
\label{cor:pc}
Let $M$ be a closed, orientable $3$-manifold, with fundamental 
group $G=\pi_1(M)$.  The following are equivalent. 
\begin{enumerate}
\item \label{cr1} 
The group $G$ is quasi-K\"{a}hler, $1$-formal, 
presentable by commutator-relators, and 
has even first Betti number. 
\item \label{cr2}
The manifold $M$ is either $S^3$ or $\#^{2k} S^1\times S^2$.
\end{enumerate}
\end{corollary}

\begin{proof}
The implication \eqref{cr2} $\Rightarrow$ \eqref{cr1} 
is clear.  So assume $G$ satisfies the conditions from \eqref{cr1}. 
By Corollary \ref{cor:cutiso}, $\corank (G)=b_1(G)=n$, with $n$ even.  
Thus, we have an epimorphism $G\surj F_n$.  Since $G$ is 
presented by commutator-relators, there is another epimorphism, 
$F_n \surj G$.  As is well known, finitely generated 
free groups are Hopfian, i.e., any epimorphism 
$F_n\surj F_n$ must be an isomorphism, 
see \cite[Theorem 2.13]{MKS}.  Hence, $G=F_n$.
Assertion \eqref{cr2} now follows from standard $3$-manifold 
theory \cite{He} and the Poincar\'{e} conjecture, 
proved by Perelman in \cite{Pe1, Pe2}.
\end{proof}

\section{Quasi-homogeneous surface singularities and $1$-formality}
\label{sect:surf}

Let $(X,0)$ be a complex, normal, isolated surface singularity. 
We may embed $(X,0)$ in $(\C^n,0)$, for some sufficiently 
large $n$. Then, for sufficiently small $\epsilon >0$, the 
singularity link, $M=X\cap S^{2n-1}_{\epsilon}$, is a closed, 
oriented, smooth $3$-manifold, whose oriented diffeomorphism 
type is independent of the choices made. 

Assume now $(X,0)$ is a quasi-homogeneous singularity. 
We may then represent $(X,0)$ by an affine surface $X$ 
with a ``good" $\C^*$-action.  By definition, this means 
the isotropy groups $\C^*_x$ are finite, for all $x\ne 0$, 
and the induced action of $\C^*$ on the finite-dimensional 
vector space $\mathfrak{m}_{(X,0)}/\mathfrak{m}_{(X,0)}^2$ 
has only positive weights. For details on this subject, we refer 
to Orlik and Wagreich \cite{OW}, Looijenga \cite[p.~175]{Lo}, 
and also \cite[p.~52, 66]{Di92}.

\begin{prop}
\label{prop:qh}
Let $(X,0)$ be a quasi-homogeneous isolated surface singularity, 
and let $M$ be the corresponding singularity link.  Then, the 
fundamental group $G=\pi_1(M)$ is a quasi-projective, 
$3$-manifold group.  Moreover,
\begin{enumerate}
\item \label{s1} 
If $b_1(M)=0$, then $M$ is formal, and so $G$ is $1$-formal.   
\item \label{s2}
If $b_1(M)>0$, then $G$ is not $1$-formal. 
\end{enumerate}
\end{prop}

\begin{proof}
Set $X^*=X \setminus \{0\}$. Clearly, $X^*$ is a smooth, 
quasi-projective variety. Moreover, the singularity link $M$ 
is homotopy equivalent to $X^*$. The first assertion follows 
at once. 

Consider now the orbit space, $C=X^*/\C^*$.  This is a 
smooth projective curve of genus $g=\frac{1}{2} b_1(M)$.  
In particular, $b_1(M)$ must be even. 

If $b_1(M)=0$, then $M$ is a rational homology $3$-sphere, i.e., 
$H^*(M,\Q)\cong H^*(S^3,\Q)$.  A result of Halperin and Stasheff 
\cite[Corollary 5.16]{HS} implies that $M$ is formal. 

Now assume $b_1(M)>0$. Then the curve $C=X^*/\C^*$ 
has genus $g\ge 1$. Let $p\colon X^* \to C$ be the projection map. 
Obviously, $p$ is surjective, and all its fibers are connected, since 
isomorphic to $\C^*$.  Therefore, $p$ is admissible (in the sense 
of \S\ref{subsec:cvqk}), and the induced homomorphism, 
$p^*\colon H^1(C,\C) \to H^1(X^*,\C)$, is an isomorphism.

By Theorem \ref{thm:res3}\eqref{r1}, the space 
$H^1(M,\C)$ is not $1$-isotropic.  Thus, $H^1(X^*,\C)=p^*H^1(C,\C)$ 
is also non-$1$-isotropic. In view of Theorem \ref{thm:posobstr}\eqref{a2}, 
the group $G=\pi_1(X^*)$ cannot be $1$-formal. 
\end{proof}

Using the above proof, we can produce examples of smooth 
affine surfaces whose fundamental groups are not $1$-formal.
With notation as above, assume $g \geq 1$. 

\begin{prop}
\label{prop:affine}
Let $f\colon X\to \C$ be a non-zero regular function with 
$f(0)=0$. Suppose the curve $V(f)=\{x \in X\mid f(x)=0\}$
intersects each fiber of the projection $p\colon X^* \to C$ in 
a finite number of points. Then $U=X \setminus V(f)$ is a 
smooth affine surface whose fundamental group is not $1$-formal.
\end{prop}

\begin{proof}
Let $A(X)$ be the coordinate ring of $X$. The coordinate 
ring of $U$ is $A(X)_f$, the localization of $A(X)$ 
with respect to the multiplicative system $\{f^k\}_{k\ge 1}$. 
Let $q\colon U \to C$ be the restriction of $p$ to $U$, 
and let $\iota\colon U\hookrightarrow X^*$ be the inclusion map. 

Assume $\pi_1(U)$ is $1$-formal.  Then, as above, 
Theorem \ref{thm:posobstr}\eqref{a2} implies that 
the subspace $q^* (H^1(C,\C))= \iota^*  (H^1(X^*,\C))$ 
is $1$-isotropic. On the other hand, by the 
aforementioned result of Sullivan \cite{Su75}, the 
cup-product is trivial on $H^1(X^*,\C)$. 
This contradiction ends the proof. 
\end{proof}

Here is an explicit family of examples. 

\begin{example}
\label{ex1}
Let $X$ be the surface in $\C^3$ given by the equation 
$x^d+y^d+z^d=0$, with $d \geq 3$.  Choose 
\begin{equation}
\label{eq:xyz}
f=x+y^2+z^3 \in A(X)=\C[x,y,z]/(x^d+y^d+z^d).
\end{equation}
It is readily verified that $V(f)\cap p^{-1}(c)$ is finite, 
for each $c\in C$.  Thus, $U=X \setminus V(f)$ is 
a smooth affine surface with $\pi_1(U)$ not $1$-formal. 
\end{example}

\section{Translated tori in characteristic varieties}
\label{sect:trans}

As in the previous section, let $X$ be an affine surface  
with a good $\C^*$-action and isolated singularity at $0$, 
$X^*=X \setminus \{0\}$, and 
$p\colon X^* \surj C$ the corresponding projection onto 
a smooth projective curve.  

Assume $C$ has genus $g\ge 1$. We then know from 
Proposition \ref{prop:qh}\eqref{s2} that the group 
$G=\pi_1(X^*)$ is not $1$-formal.  Of course, 
this prevents us from using Theorem \ref{thm:tcone} in 
computing the characteristic variety $V_1(G)$.  Nevertheless, 
such a computation can still be carried out 
in this situation, using different techniques. 

\subsection{Pontrjagin duality}
\label{subsec:pont}

First, we need some additional notation. Let $\Tors$ 
be the torsion part of the homology group $H_1(X^*,\Z)$. Since
$p^*\colon H^1(C,\C) \to H^1(X^*,\C)$ is an isomorphism, the 
group $\Tors$ is the kernel of  $p_*\colon H_1(X^*,\Z) \to H_1(C,\Z)$. 
Let $i_*\colon H_1(\C^*,\Z) \to H_1(X^*,\Z)$ be the morphism 
induced by the inclusion of a generic fiber of $p$ into $X^*$. 
It is clear that the image of $i_*$ is a cyclic subgroup in $\Tors$, 
that is, 
\begin{equation}
\label{eq:hh}
\im (i_*)=\langle h\rangle, \ \text{for some $h \in \Tors$}.
\end{equation}

The direct sum decomposition $H_1(X^*,\Z)= \Z^{2g} \oplus \Tors$ 
yields a corresponding decomposition 
of the character group $\T(G)=\Hom(G,\C^*)$,
\begin{equation}
\label{eq:pont}
\T(G)=(\C^*)^{2g} \times \hat{\Tors},
\end{equation}
where $\hat{\Tors} =\Hom(\Tors,\C^*)$ is the Pontrjagin dual of $\Tors$. 
In this decomposition, the identity component $\T^0(G)$ 
corresponds to the algebraic torus $(\C^*)^{2g} \times 1$.

Following \cite{Di07}, let us associate to the admissible map 
$p\colon X^* \to C$ the finite abelian group
\begin{equation}
\label{eq:tp}
T(p)=\Tors/\langle h\rangle.
\end{equation}
The Pontrjagin dual $\hat{T}(p)$ may be regarded as 
a subgroup of $\hat{\Tors} $, and hence of $\T(G)$.  Plainly,  
the index $[\hat{\Tors} : \hat{T}(p)]$ equals $\ord(h)$, the order 
of the element $h$ in $\Tors$.

\subsection{Parametrizing translated components}
\label{subsec:trans}
We are now ready to state the main result of this section. 

\begin{prop}
\label{propN}
The positive-dimensional irreducible components of $V_1(X^*)$ 
consist of all the translates of the identity component $\T^0(G)$ 
by the elements of $\hat{T}(p)$ if $g>1$, and by the elements of 
$\hat{T}(p)\setminus \{1\}$ if $g=1$. 
\end{prop}

\begin{proof}
As Arapura showed in \cite{Ar}, any positive-dimensional,  
irreducible component $W$ of $V_1(X^*)$ corresponds to 
a surjective morphism $f\colon X^* \to S$ with connected 
generic fiber, where $S$ is a smooth curve with $\chi(S) \leq 0$. 
As shown in \cite{Di07}, the components associated to such 
a morphism $f$ are parametrized by $\hat T(f)$ when 
$\chi(S) < 0$, and by $\hat T(f) \setminus \{ 1\}$ when $\chi(S) = 0$. 
It thus remains to show that $p\colon X^* \to C$ is the only 
morphism with the properties of $f$ above, up to isomorphisms 
of $C$.

Assume first that $S$ is a curve of genus $h>0$, i.e., the 
projective smooth closure $\overline{S}$ of $S$ has genus $h>0$.
Then each fiber of $p$, being isomorphic to $\C^*$, is mapped 
to a point by $f$ (the only morphisms $\CP^1 \to \overline{S}$
are the constant ones!). We may then factor $f$ as 
$X^*\xrightarrow{p} C \xrightarrow{q}  S$. 
This shows that $S=\overline{S}$.  Moreover, since the  
generic fibers of $f$ are connected, the morphism $q\colon C \to S$ 
has degree $1$, i.e., it is an isomorphism.

Assume now $S$ is a rational curve. Since $\chi(S)\le 0$, the 
curve $S$ is obtained from $\CP^1$ by deleting at least $2$ points. 
Hence, the mixed Hodge structure on $H^1(S,\Q)$ is pure of 
weight $2$. On the other hand, the mixed Hodge structure on 
$H^1(X^*,\Q)$ is pure of weight $1$, since 
$p^*\colon H^1(C,\Q) \to H^1(X^*,\Q)$ is an isomorphism. 
Hence, $f^*\colon H^1(S,\Q) \to H^1(X^*,\Q)$ cannot be injective, 
a contradiction.
\end{proof}

\subsection{Non-formality and the tangent cone formula}
\label{subs:nftcone}
Comparing the conclusions of Propositions \ref{prop:qh} 
and \ref{propN} reveals a rather subtle phenomenon: 
the tangent cone formula from Theorem \ref{thm:tcone} 
may hold, even when the group $G$ is not $1$-formal. 

To see why this is the case, let us first compute the analytic 
germs at the origin for the cohomology jumping loci of
$X^*=X \setminus \{ 0\}$. By Proposition \ref{propN}, 
\begin{equation}
\label{eq:v1x}
(V_1(X^*),1)=
\begin{cases}
(\{ 1\}, 1)& \text{if $g=1$},\\[3pt]
((\C^*)^{2g}, 1)& \text{if $g>1$}.\\
\end{cases}
\end{equation}
On the other hand, $X^*$ is homotopy equivalent to the singularity link $M$.
It follows from Theorem \ref{thm:res3}\eqref{r2} that
\begin{equation}
\label{eq:r1x}
R_1(X^*)=H^1(X^*, \C)= \C^{2g}, \quad
\text{for all $g\ge 1$}. 
\end{equation}

\begin{corollary}
\label{cor:nftcone}
Let $G=\pi_1(X^*)$ as above, and assume $g>1$.  
Then $G$ is not $1$-formal, yet the tangent cone formula 
$TC_1(V_1(G))=R_1(G)$ holds. 
\end{corollary}

\begin{proof}
Since $g\ge 1$, Proposition \ref{prop:qh}\eqref{s2} 
insures that $G$ is not $1$-formal.  
Using formulas \eqref{eq:v1x} and \eqref{eq:r1x}, we see 
that the analytic isomorphism from Theorem \ref{thm:tcone},
$\exp \colon (R_1(G), 0) \isom (V_1(G), 1)$, 
holds. In particular, $TC_1(V_1(G))=R_1(G)$.  
\end{proof}

\subsection{Seifert invariants}
\label{subsec:seifert}
As above, let $M$ be the $3$-manifold associated to the 
quasi-homogenous singularity $(X,0)$. The $S^1$-equivariant 
diffeomorphism type of $M$ is determined by the following 
Seifert invariants associated to the projection 
$\left.p \right|_{M}\colon M\to C$:
\begin{itemize}
\item The exceptional orbit data,   
$((\alpha _1,\beta _1), \dots , (\alpha_m,\beta_m))$, 
with $\alpha_i, \beta_i\in \Z$, $\alpha_i> 1$, and 
$\gcd(\alpha_i,\beta_i)=1$.
\item The genus $g=g(C)$ of the base curve $C=M/S^1$, 
with $g\ge 0$.
\item The Euler number $e$  of the Seifert fibration, 
with $e\in \Q$, and $e<0$. 
\end{itemize}
For full details, see \cite{Ne1, Ne2, NR, OW}. 

Since $e\ne 0$, we have $H_1(M,\Z)=\Z^{2g}\oplus \Tors$, 
with $\Tors$ finite, as in \S\ref{subsec:pont}. Let $h\in \Tors$ be 
the element defined by \eqref{eq:hh}. Using the proof 
of Theorem 4.1 from \cite{NR}, we compute
\begin{equation}
\label{eq:seifert}
\abs{\Tors}=\alpha_1\cdots \alpha_m \cdot \abs{e}, \quad 
\ord (h)=\lcm (\alpha_1,\dots ,\alpha_m) \cdot \abs{e}.
\end{equation}
It follows that the group $T(p)=\Tors/\langle h\rangle$ from 
\eqref{eq:tp} has order 
\begin{equation}
\label{eq:alpha}
\alpha=\alpha_1\cdots \alpha_m / \lcm(\alpha_1,\dots ,\alpha_m).
\end{equation}

\begin{remark}
\label{rem:multiple}
The morphism $p\colon X^* \to C$ has exactly $m$ multiple fibers, 
each isomorphic to $\C^*$, and with multiplicities $\alpha_1,\dots ,\alpha_m$.
Thus, formula \eqref{eq:alpha} for the order of $T(p)$ may alternatively 
be obtained from \cite[Theorem 5.3(i)]{Di07}.
\end{remark}

Let $G=\pi_1(M)$.  Recall that the character variety $\T(G)$ can be 
identified with $\T^0(G)\times \hat{\Tors}$, a disjoint union 
of $\abs{\Tors}$ copies of the algebraic torus $\T^0(G)=(\C^*)^{2g}$. 
Proposition \ref{propN} now yields the following corollary.
\begin{corollary}
\label{cor:v1seifert}
The positive-dimensional components of $V_1(M)$ 
are: $\T^0(G)$ if $g>1$, and $\alpha-1$  
translated copies of $\T^0(G)$, for any $g\ge 1$. 
\end{corollary}

In particular, $V_1(M)$ contains positive-dimensional 
components not passing through the origin if and only if 
$\alpha\ne 1$.

\begin{example}
\label{ex:bundle}
When $M\to C$ is a bona-fide circle bundle, with Euler 
class $e\in \Z$, we may take $m=0$, in which case   
$\abs{\Tors}=\ord(h)=\abs{e}$.  Thus, $V_1(M)$ has 
positive-dimensional components only if $g>1$---and 
then the only such component is $\T^0(G)$.
\end{example} 

\subsection{Brieskorn manifolds}
\label{subsec:brieskorn}
Let $(a_1,\dots ,a_n)$ be an $n$-tuple of integers, 
with $a_j\ge 2$. Consider the variety $X$ in $\C^n$ 
defined by the equations
$c_{j1}x_1^{a_1}+\cdots +c_{jn}x_n^{a_n}=0$, for $1\le j\le n-2$. 
Assuming the matrix of coefficients $(c_{jk})$ has all maximal 
minors non-zero, $X$ is a quasi-homogeneous surface, whose 
singularity link at $0$ is the well-known Brieskorn manifold 
$M=\Sigma(a_1,\dots, a_n)$. 

Set $\ell=\lcm (a_1,\dots, a_n)$,  
$\ell_j=\lcm (a_1,\dots, \widehat{a_j}, \dots, a_n)$, 
and $a=a_1\cdots a_n$. Computations from \cite{NR, Ne1} 
show that the Seifert invariants of $M$ are:
\begin{itemize}
\item $(s_1(\alpha_1,\beta_1),\dots,s_n(\alpha_n,\beta_n))$, 
with $\alpha_j= \ell/\ell_j$,  $\beta_j \ell\equiv a_j\, 
(\operatorname{mod}\, \alpha_j)$, and 
$s_j=a/(a_j \ell_j)$, where $s_j(\alpha_j,\beta_j)$ signifies $(\alpha_j,\beta_j)$ 
repeated $s_j$ times, unless $\alpha_j=1$, in which case $s_j(\alpha_j,\beta_j)$ 
is to be removed from the list.
\item $g=\frac{1}{2}\big(2+(n-2) a/\ell -\sum_{j=1}^{n} s_j \big)$.
\item $e=-a/\ell^2$. 
\end{itemize}

The case $n=3$ was studied in detail by Milnor in \cite{Mi}. 
The simplest situation is when $\ell_1=\ell_2=\ell_3=\ell$, in   
which case $\alpha_j=1$, for all $j$. Therefore, 
$M=\Sigma(a_1,a_2,a_3)$ 
is a smooth circle bundle over $\Sigma_g$, with Euler 
number $e$ as above, and the positive-dimensional components  
of $V_1(M)$ are as described in Example \ref{ex:bundle}. 

\begin{example}
\label{ex:nilpotent}
As noted by Milnor \cite{Mi}, the manifolds $M_1=\Sigma(2,3,6)$, 
$M_2=\Sigma(2,4,4)$, and $M_3=\Sigma(3,3,3)$ are 
all nilmanifolds---in fact, circle bundles over the torus, 
with Euler number $1$, $2$, and $3$, respectively.  
We know from the above that $V_1(M_i)$ has no 
positive-dimensional components. More is true, though: 
as shown in \cite[Theorem 1.1]{MP}, $V_1(M_i)=\{1\}$.   
In other words, there are no isolated points 
in $V_1(M_i)$, either, besides the identity. 
\end{example}

In general, though, the first characteristic variety of a Brieskorn 
manifold will contain translated components.
 
\begin{example}
\label{ex:336}
The manifold $M=\Sigma(3,3,6)$ has Seifert invariants 
$((2,1), (2,1), (2,1))$, $g=1$, and $e=-3/2$.  
By Corollary \ref{cor:v1seifert}, $V_1(M)$ has 
$3$ positive-dimensional irreducible components, 
all of dimension $2$, none of which passes through 
the identity.
\end{example}

\begin{ack}
Much of this work was done during visits at Universit\'{e} de 
Nice--Sophia Antipolis by A.~S. (September, 2007) and S.~P. 
(June 2008).  Both these authors thank the Laboratoire 
Jean A. Dieudonn\'{e} for its support and hospitality during 
their stay in Nice, France. Additionally, A.~D. is grateful to 
ASSMS, Government College University, Lahore, Pakistan, 
where part of the work on this paper was done.
\end{ack}

%\vspace*{-3pc}

\newcommand{\arxiv}[1]
{\texttt{\href{http://arxiv.org/abs/#1}{arXiv:#1}}}

\renewcommand{\MR}[1]
{\href{http://www.ams.org/mathscinet-getitem?mr=#1}{MR#1}}


\begin{thebibliography}{00}

\bibitem{Ar} D.~Arapura, 
{\em Geometry of cohomology support loci for local systems 
{\rm I}}, J. Alg. Geometry \textbf{6} (1997), no.~3, 563--597.  
\MR{1487227}

\bibitem{Ba} I.~Bauer, 
{\em Irrational pencils on non-compact algebraic manifolds}, 
Internat. J. Math. \textbf{8} (1997), no.~4, 441--450. 
\MR{1460895}

\bibitem{Be} A.~Beauville, 
{\em Annulation de $H^1$ et syst\`{e}mes paracanoniques 
sur les surfaces}, J. reine angew. Math. \textbf{388} (1988), 
149--157.  
\MR{0944188}

\bibitem{BP} B.~Berceanu, S.~Papadima, 
{\em Cohomologically generic $2$-complexes and 
$3$-dimensional Poincar\'{e} complexes},  
Math. Ann. \textbf{298} (1994), no.~3, 457--480. 
\MR{1262770}

\bibitem{Ca} F.~Catanese, 
{\em Moduli and classification of irregular Kaehler manifolds 
(and algebraic varieties) with Albanese general type fibrations}, 
Invent. Math. \textbf{104} (1991), no.~2, 263--289. 
\MR{1098610}

\bibitem{Ch} K.-T.~Chen,
{\em Extension of $C^{\infty}$ function algebra by
integrals and {M}alcev completion of $\pi_1$}, 
Adv. in Math. \textbf{23} (1977), no.~2, 181--210.
\MR{0458461}

\bibitem{CS08} D.~Cohen, A.~Suciu,
{\em The boundary manifold of a complex line arrangement}, 
Geometry \& Topology Monographs \textbf{13} (2008), 105--146.
\MR{2508203}

\bibitem{DGMS}  P.~Deligne, P.~Griffiths, J.~Morgan, D.~Sullivan,
{\em Real homotopy theory of {K}\"{a}hler manifolds},
Invent. Math. \textbf{29} (1975), no.~3, 245--274.
\MR{0382702}

\bibitem{Di92} A.~Dimca, 
{\em Singularities and topology of hypersurfaces}, 
Universitext, Springer Verlag, New York, 1992.
\MR{1194180}

\bibitem{Di07} A.~Dimca, 
{\em Characteristic varieties and constructible sheaves}, 
Rend. Lincei Mat. Appl. \textbf{18} (2007), no.~4, 365--389.
\MR{2349994} 

\bibitem{Di08}  A.~Dimca, 
{\em On the isotropic subspace theorems}, 
Bull. Math. Soc. Sci. Math. Roumanie, \textbf{51} (2008), 
no.~4, 307--324. 
\MR{2447175}

\bibitem{DPS07} A.~Dimca, S.~Papadima, A.~Suciu,
{\em Alexander polynomials: Essential variables and 
multiplicities}, Int. Math. Res. Notices \textbf{2008}, 
no.~3, Art. ID rnm119, 36 pp. 
\MR{2416998}


\bibitem{DPS-jump} A.~Dimca, S.~Papadima, A.~Suciu,
{\em Topology and geometry of cohomology jump loci}, 
Duke Math. Journal \textbf{148} (2009), no.~3, 405--457.
\MR{2527322}

\bibitem{DS07} A.~Dimca,  A.~Suciu,
{\em Which $3$-manifold groups are K\"{a}hler groups?}, 
J. Eur. Math. Soc. \textbf{11} (2009), no.~3, 521--528.
\MR{2505439}

\bibitem{GL} M.~Green, R.~Lazarsfeld, 
{\em Higher obstructions to deforming cohomology groups 
of line bundles}, J. Amer. Math. Soc. \textbf{4} (1991), 
no.~1, 87--103.
\MR{1076513}

\bibitem{HS} S.~Halperin, J.~Stasheff, 
{\em Obstructions to homotopy equivalences}, 
Adv. in Math. \textbf{32} (1979), no.~3, 233--279. 
\MR{0539532}

\bibitem{Ha} S.~Harvey,
{\em On the cut number of a $3$-manifold}, Geom. Topol. 
\textbf{6} (2002), 409--424.
\MR{1928841}

\bibitem{He}  J.~Hempel,
{\em $3$-Manifolds}, Ann. of Math. Studies, no.~86, 
Princeton Univ. Press, Princeton, NJ, 1976. 
\MR{0415619}

\bibitem{Ko} T.~Kohno,
{\em On the holonomy {L}ie algebra and the nilpotent completion
of the fundamental group of the complement of hypersurfaces},
Nagoya Math. J. \textbf{92} (1983), 21--37.  
\MR{0726138} 

\bibitem{LR} C.~Leininger, A.~Reid, 
{\em The co-rank conjecture for $3$-manifold groups},  
Algebr. Geom. Topol. \textbf{2} (2002), 37--50. 
\MR{1885215}  

\bibitem{Li}  A.~Libgober,
{\em First order deformations for rank one local systems 
with a non-vanishing cohomology}, Topology Appl. 
\textbf{118} (2002), no.~1-2, 159--168. 
\MR{1877722}

\bibitem{Lo} E.~J.~N.~Looijenga,
{\em Isolated singular points on complete intersections}, 
London Math. Soc. Lecture Note Series, vol.~77, 
Cambridge University Press, Cambridge, 1984.
\MR{0747303}

\bibitem{MP} A.~Macinic, S.~Papadima,
{\em Characteristic varieties of nilpotent groups and 
applications}, in: Proceedings of the Sixth Congress 
of Romanian Mathematicians (Bucharest, 2007), 
pp.~57--64, vol.~1, Romanian Academy, Bucharest, 2009.

\bibitem{MKS} W.~Magnus, A.~Karrass, D.~Solitar,
{\em Combinatorial group theory} (2nd ed.), 
Dover, New~York, 1976.
\MR{0422434}

\bibitem{McM} C.~T.~McMullen,
{\em The Alexander polynomial of a $3$-manifold and the 
Thurston norm on cohomology}, Ann. Sci. \'{E}cole Norm. Sup. 
\textbf{35} (2002), no.~2, 153--171. 
\MR{1914929}.

\bibitem{Mi} J.~Milnor, 
{\em On the $3$-dimensional Brieskorn manifolds $M(p,q,r)$}, 
in: {\em Knots, groups, and $3$-manifolds}, pp.~175--225, Ann. of 
Math. Studies, no.~84, Princeton Univ. Press, Princeton, NJ, 1975. 
\MR{0418127}

\bibitem{M}  J.~W.~Morgan,
{\em The algebraic topology of smooth algebraic varieties},
Inst. Hautes \'{E}tudes Sci. Publ. Math. \textbf{48} (1978), 137--204.
\MR{0516917} 

\bibitem{Ne1} W.~Neumann,
{\em Brieskorn complete intersections and automorphic forms}, 
Invent. Math. \textbf{42} (1977), 285--293. 
\MR{0463493}

\bibitem{Ne2} W.~Neumann,
{\em Abelian covers of quasihomogeneous surface singularities}, 
in: Singularities, Part 2 (Arcata, Calif., 1981), 233--243, 
Proc. Sympos. Pure Math., vol.~40, Amer. Math. Soc., 
Providence, RI, 1983.
\MR{0713252}

\bibitem{NR} W.~Neumann, F.~Raymond, 
{\em Seifert manifolds, plumbing, $\mu$-invariant and 
orientation reversing maps}, Lecture Notes in Math., 
vol.~664, Springer, Berlin, 1978, pp.~163--196. 
\MR{0518415} 

\bibitem{OW} P.~Orlik, P.~Wagreich, 
{\em Equivariant resolution of singularities with $\C^*$-action},  
Lecture Notes in Math., vol.~298, Springer, Berlin, 1972, 
pp.~270--290. 
\MR{0393029}

\bibitem{PS-chenlie} S.~Papadima, A.~I.~Suciu,
{\em Chen {L}ie algebras}, International Math.~Research 
Notices  \textbf{2004} (2004), no.~21, 1057--1086.  
\MR{2037049}  

\bibitem{Pe1}  G.~Perelman,
{\em The entropy formula for the Ricci flow and its 
geometric applications}, 
\arxiv{math.DG/0211159}

\bibitem{Pe2}  G.~Perelman,
{\em Ricci flow with surgery on three-manifolds}, 
\arxiv{math.DG/0303109}

\bibitem{Qu} D.~Quillen,
{\em Rational homotopy theory}, Ann. of Math.
\textbf{90} (1969), 205--295.  
\MR{0258031}

\bibitem{Re}  A.~Reznikov,
{\em The structure of {K}\"{a}hler groups. \textup{I}. 
Second cohomology}, in: Motives, polylogarithms 
and Hodge theory, Part II (Irvine, CA, 1998), pp.~717--730, 
Int. Press Lect. Ser., vol.~3, II, Int. Press, Somerville, MA, 2002. 
\MR{1978716}

\bibitem{Si}  A.~Sikora,
{\em Cut numbers of $3$-manifolds}, Trans. Amer. 
Math. Soc. \textbf{357} (2005), no.~5, 2007--2020. 
\MR{2115088}

\bibitem{Sp92} C.~Simpson, 
{\em Higgs bundles and local systems}, Inst. Hautes 
\'{E}tudes Sci. Publ. Math. \textbf{75} (1992), 5--95. 
\MR{1179076}

\bibitem{Su75}  D.~Sullivan,
{\em On the intersection ring of compact three manifolds}, 
Topology \textbf{14} (1975), no.~3, 275--277. 
\MR{0383415}

\bibitem{Su77}  D.~Sullivan,
{\em Infinitesimal computations in topology},
Inst. Hautes \'{E}tudes Sci. Publ. Math.
\textbf{47} (1977), 269--331.
\MR{0646078} 

\end{thebibliography}
\end{document}